\documentclass[10pt,oneside,a4paper]{amsart}
\usepackage{latexsym}
\usepackage{amsfonts,amsmath,amssymb,indentfirst}
\usepackage[english]{babel}
\usepackage[all]{xy}
\usepackage[pdftex]{hyperref}

\newcommand{\bdism}{\begin{displaymath}}
\newcommand{\edism}{\end{displaymath}}
\newcommand{\cc}{\mathbb{C}}
\newcommand{\rr}{\mathbb{R}}

\newcommand{\pp}{\mathbb{P}}
\newcommand{\oo}{\mathcal{O}}

 \DeclareMathOperator{\Chow}{Chow}

\newtheorem{theorem}{Theorem}[section]

\newtheorem{corollary}[theorem]{Corollary}

\address{Department of Mathematics, Columbia University, New York NY 10027,
USA} \email{dicerbo@math.columbia.edu}

\address{Department of Mathematics, Duke University, Durham NC 27708-0320,
USA} \email{luca@math.duke.edu}

\author{Gabriele Di Cerbo and Luca F. Di Cerbo}

\title{\bf A sharp cusp count for complex hyperbolic surfaces and related results}
\begin{document}

\begin{abstract}

We derive a sharp cusp count for finite volume complex hyperbolic surfaces
which admit smooth toroidal compactifications. We use this result, and the techniques developed in \cite{DiCerbo}, to study the geometry of cusped complex hyperbolic surfaces and their compactifications.

\end{abstract}
\maketitle

\tableofcontents

\section{Introduction}
\pagestyle{myheadings}
\markboth{Right}{\textbf{A SHARP CUSP COUNT FOR COMPLEX HYPERBOLIC SURFACES}}
\pagenumbering{arabic}

The problem of counting the number of ends in complete pinched negatively curved finite volume manifolds is a long standing problem in classical differential geometry. For real and complex hyperbolic manifolds, geometric topology techniques have been widely and successfully applied. The literature on the subject is enormous, we thus refer the interested reader to the classical paper of Parker \cite{Parker} and to the more recent work of Stover \cite{Stover} for an overview.

For cusped complex hyperbolic manifolds, techniques coming from algebraic geometry can also be applied. The key for this approach is the existence of particularly nice compactifications of such complex hyperbolic manifolds, known as toroidal compactifications. We recently applied this circle of ideas in \cite{DiCerbo}, where many of the  finiteness properties of cusped hyperbolic manifolds and their compactifications were made effective.   

The purpose of this paper is to improve, in complex dimension two, the results contained in \cite{DiCerbo}. The main novelty contained here is that some of the results are now \emph{sharp}. This is quite unusual even in low dimensions and it seems to be the first instance where sharpness is actually achieved.\\

\noindent\textbf{Acknowledgements}. We thank the organizers of the conference ``Algebraic \&
Hyperbolic Geometry - New Connections'' for stimulating our interest in these topics.

\subsection{Preliminaries and main results}

Let $\mathcal{H}^{2}$ be the two dimensional complex hyperbolic space. Given a torsion free non-uniform lattice $\Gamma$ in $\textrm{PU}(1,2)$, then $\mathcal{H}^{2}/\Gamma$   
is a non-compact finite volume complex hyperbolic surface with cuspidal ends. The cusps of  $\mathcal{H}^{2}/\Gamma$ are then in one to one correspondence with the conjugacy classes of maximal parabolic subgroups in $\Gamma$. It is known that when the parabolic elements in $\Gamma$ have no rotational part (e.g. if $\Gamma$ is neat), then $\mathcal{H}^{2}/\Gamma$ admits a compactification $(X, D)$ consisting of a smooth projective surface and an exceptional divisor $D$. Recall that each maximal parabolic subgroup can be thought as a lattice in $\textrm{N}\rtimes \textrm{U}(1)$ where $\textrm{N}$ is a Heisenberg type Lie group, and that a parabolic isometry is said to have no rotational part if it has no $\textrm{U}(1)$ component. The pair $(X, D)$ is known as the toroidal compactification of $\mathcal{H}^{2}/\Gamma$. For a more detailed discussion of these topics, we refer to the classical references \cite{Ash}, \cite{Ji}. For a quick introduction, the interested reader may also refer to Section 1.1 in \cite{DiCerbo} and the bibliography therein. 

Let us recall a few important geometric features of $(X, D)$. First, the exceptional divisor $D$ consists of disjoint smooth elliptic curves with negative self-intersection which are in one to one correspondence with the cusps of $X\backslash D$. Second, the pair $(X, D)$ is of logarithmic general type and $D$-minimal, i.e., the log-canonical divisor $K_{X}+D$ has maximal Kodaira dimension and $X$ does not contain any exceptional curve $E$ of the fist kind such that $E\cdot D\leq 1$. Finally, because of the Hirzebruch-Mumford proportionality theorem \cite{Mumford}, we have 
\begin{align}\notag
3\overline{c}_{2}=\overline{c}^{2}_{1}
\end{align}
where $\overline{c}^{2}_{1}$ and $\overline{c}_{2}$ denote the logarithmic Chern numbers of the compactification $(X, D)$. Recall that $\overline{c}^{2}_{1}$ is the self-intersection of $K_{X}+D$, while $\overline{c}_{2}$ is the topological Euler characteristic of $X\backslash D$. Since $X\backslash D$ is by construction negatively curved, Gromov-Harder's generalization of Chern-Gauss-Bonnet \cite{Gro1} implies that $\overline{c}_{2}$ must be a strictly positive integer.

In \cite{DiCerbo}, we showed that, in any dimension, $K_{X}+D$ is not only big and nef but it is a limit of ample divisors of the form $K_{X}+\alpha D$ where $\alpha$ is a real number strictly less than one. Let us recall the technical form of this result as it will be used here.

\begin{theorem}[\cite{DiCerbo}, Theorem 1.1]\label{Gap}
Let $(X, D)$ be a toroidal compactification. Then $K_{X}+\alpha D$ is ample for any $\alpha\in(\frac{1}{3}, 1)$.
\end{theorem}

As extensively shown in \cite{DiCerbo}, this ``gap'' theorem has many consequences for the theory of cusped complex hyperbolic manifolds. In fact, one of the first applications is a cusp count in terms of the volume, valid in any dimension, which turned out to be the best count available in the literature for surfaces and threefolds, see Theorem 1.5 in \cite{DiCerbo}. 

The main purpose of this note is to show that, in complex dimension two, such cusp count can not only be improved but also made sharp. Recall that, if we normalize the holomorphic sectional curvature of $X\backslash D$ to be $-1$, we have
\begin{align}\notag
\textrm{Vol}_{-1}(X\backslash D)=\frac{8}{9}\pi^{2}(K_{X}+D)^{2}
\end{align}
where by $\textrm{Vol}_{-1}(X\backslash D)$ we indicate the Riemannian volume of $X\backslash D$ computed with respect to the normalized locally symmetric metric. We then state the cusp count in terms of the self-intersection of the log-canonical divisor of the compactification $(X, D)$.

\begin{theorem}\label{Sharp}
Let $(X^{2},D)$ be a toroidal compactification. Let $q$ be the
number of components of $D$. Then \bdism q\leq
\frac{4}{3}(K_{X}+D)^{2}. \edism
Moreover, there exists a smooth toroidal compactification which saturates this bound.
\end{theorem}

Despite the existence of many linear upper bounds on the number of cusps of a locally symmetric space in terms of its volume, Theorem \ref{Sharp} appears to be the first sharp bound. Unfortunately, Theorem \ref{Sharp} holds in complex dimension two only and it is not clear how to obtain sharp bounds in higher dimensions. On the other hand, it seems an extremely complicated problem to sharpen any ``effective'' technique, either in geometric topology or complex/differential geometry, in higher dimensions. To give an example of the difficulty of this problem, it will suffice to recall that we are missing an \emph{explicit} example of a smooth threefold which is the toroidal compactification of a cusped complex hyperbolic threefold. Of course, for group theoretical reasons we know that such objects exist in great profusion. Nevertheless, we still seem to miss a technique to systematically produce explicit examples. Fortunately, Hirzebruch in \cite{Hir1} constructed many explicit examples in complex dimension two, see Section \ref{Main} for details. These examples play a crucial role in testing the sharpness of Theorem \ref{Sharp}.

Next, using Theorems \ref{Gap} and \ref{Sharp} we can study in details the geometry of cusped complex hyperbolic surfaces and their compactifications. The results which follows substantially improve, in complex dimension two, the ones contained in \cite{DiCerbo}. First, we describe the pair $(X, D)$ when realized as a sub-variety in some projective space. 

\begin{theorem}\label{Degree}
Let $(X^{2},D)$ be a toroidal compactification. Let $L:=K_{X}+D$ and let us denote by $D_{i}$ the components of $D$. Then $7L-4D$ is very ample. In particular, $7L-4D$ embeds $X$ as a
smooth subvariety of $\pp^{5}$ of degree $d \leq 49 L^{2}$. Moreover, if we let $d_{i}$ be the degree of $D_{i}$ under the embedding given by
$7L-4D$, then  $d_{i}<6 L^{2}$.
\end{theorem}

The importance of this result relies on the fact that it can be used to derive an explicit bound on the number of complex hyperbolic surfaces with given upper bound on the volume. More precisely, Theorem \ref{Degree} can be combined with classical Chow varieties techniques to derive an effective version of the classical Wang's finiteness theorem, \cite{Wan}. For the numerical value of this bound we refer to Corollary \ref{Count} in Section \ref{Applications}.

Second, we can construct very explicitly the singular compactifications (Baily-Borel-Siu-Yau) of these ball quotients by studying the pluri-log-canonical maps of their toroidal compactifications. 

\begin{theorem}\label{Baily1}
Let $(X^{2}, D)$  be a toroidal compactification. Then for all
$m\geq 3$ the map associated to $|m(K_{X}+D)|$ is a morphism which
maps all the different components of $D$ to distinct singular
points. Furthermore, for all $m\geq 4$ this morphism defines an
embedding of $X\backslash D$ into some projective space.
\end{theorem}

Theorem \ref{Baily1} does not only improve Theorem 1.3 in \cite{DiCerbo} when the dimension is equal to two, but also Theorem B in \cite{Luca1} which appears to be the most accurate result currently available in the literature.

Finally, we present two sided bounds on the Picard numbers of toroidal compactifications in terms of the number of cusps and the volume. Interestingly, we are able to show that the lower bound is sharp.

\begin{theorem}\label{Picard}
Let $(X^{2},D)$ be a toroidal compactification. Let $q$ be the number of
components of $D$ and let us denote by $\rho(X)$ the Picard
number of $X$. Then 
\begin{align}\notag
q+1\leq \rho(X)\leq 4+\frac{13}{3}(K_{X}+D)^{2}.
\end{align}
Moreover, there exists a smooth toroidal compactification which saturates the lower bound.
\end{theorem}

In Section 3.4 of \cite{DiCerbo}, we computed an explicit upper bound on the Picard number of a toroidal compactification in terms of the top self-intersection of its log-canonical divisor. Unfortunately, the techniques presented in Section 3.4 of \cite{DiCerbo} are applicable in dimension bigger or equal than three and they do not generalize to the case of surfaces. On the other hand, Theorem \ref{Degree} and Corollary \ref{Count} are enough to theoretically imply the existence of such an upper bound in the two dimensional case. The explicit upper bound given in Theorem \ref{Picard} fixes this gap.

\section{A sharp cusp count}\label{Main}

In this section, we present a sharp cusp count for complex hyperbolic surfaces which admit smooth toroidal compactifications. More precisely, we give a proof of Theorem \ref{Sharp} stated in the Introduction. The strategy of the proof relies on the fact that we are working in complex dimension two. In fact, the Kodaira-Enriques classification \cite{Van de Ven} of compact complex surface is crucially used. The other important ingredients are the Hirzebruch-Mumford proportionality and, in order to prove sharpness, an explicit toroidal compactification which can be extracted from the work of Hirzebruch in \cite{Hir1}.

\begin{proof}[Proof of Theorem \ref{Sharp}]
Given a two dimensional toroidal compactification $(X, D)$, let us observe that the holomorphic Euler characteristic, say $\chi(\oo_{X})$, of $X$ must be nonnegative. In fact, since
$\chi(\oo_{X})$ is a birational invariant, by the Kodaira-Enriques classification we know that the only surfaces with negative
holomorphic Euler characteristic are birational to ruled surfaces over Riemann
surfaces with genus greater or equal than two. It remains to show that such surfaces cannot be toroidal compactifications. The key observation is that they cannot contain any smooth elliptic curve. Assume not, then a smooth elliptic curve must be a section of the ruling. We then obtain a surjective morphism from the elliptic curve to the base of the ruling which we are assuming to be of genus greater or equal than two. This contradicts the classical genus formula of Riemann-Hurwitz \cite{Harris}.

Thus, by Noether's formula, see page 9 in \cite{Friedman}, and
Hirzebruch-Mumford proportionality we conclude that
\begin{align}\notag
c^{2}_{1}+c_{2}=\overline{c}^{2}_{1}+D^{2}+\frac{1}{3}\overline{c}^{2}_{1}\geq
0
\end{align}
which then implies
\begin{align}\notag
\frac{4}{3}\overline{c}^{2}_{1}\geq -D^{2}\geq q.
\end{align}
It remains to show that this bound is sharp. To this aim, recall the construction of a four cusps complex hyperbolic surface with normalized Riemannian volume equal to $\frac{8}{3}\pi^{2}$. Let $Y=\cc^{2}/\Gamma$ be a product Abelian surface, with $\Gamma=\Lambda_{1, \tau}\times\Lambda_{1, \tau}$ where $\Lambda_{1, \tau}$ is the lattice in $\cc$ generated by $1$ and $\tau=e^{\frac{\pi i}{3}}$. In $Y$, let us consider a curve $C$ defined by the equations
\begin{align}\notag
w=0, \quad z=0, \quad w=z, \quad w=\tau z
\end{align}
where $(w, z)$ are the natural product coordinates on $Y$. The divisor $C$ consists of four smooth disjoint elliptic curves intersecting at the point $p=(0, 0)$ only. Given $(Y, C)$, by blowing up the point $p$ we obtain a pair $(X, D)$, where $X$ is the blow up at $p$ of Y and $D$ the proper transform of $C$ which consits of four disjoint smooth elliptic curves. Now the pair $(X, D)$ can be shown to be a toroidal compactification, see \cite{Hir1} and Section 4 in \cite{Classification}. Thus, since 
\begin{align}\notag
(K_{X}+D)^{2}=K^{2}_{X}-D^{2}=3,
\end{align}
the bound is proven to be sharp and the proof is complete.
\end{proof}

Note that the proof of Theorem \ref{Sharp} gives also a bound of the self-intersection
of $D$ in terms of the self-intersection of the log-canonical divisor of the pair $(X, D)$. This fact is important when trying to apply Chow varieties techniques in complex dimension two. Thus, we summarize this bound into a corollary as it will be used in Section \ref{Applications}.

\begin{corollary}\label{topsurf}
Let $(X^{2}, D)$ be a toroidal compactification. Then 
\begin{align}\notag
0<-D^{2}\leq \frac{4}{3}(K_{X}+D)^{2}.
\end{align}
\end{corollary}

\section{Applications}\label{Applications}

In this section, we substantially improve the theory developed \cite{DiCerbo} in the special case of surfaces. In particular, we give the proof of Theorem \ref{Degree}, \ref{Baily1} and \ref{Picard} stated in the introduction. These improvements not only rely on Theorem \ref{Sharp}, but also on the fact that the birational geometry of
adjoint linear systems on surfaces is completely understood thanks
to the theorem of Reider. Let us recall the statement of this
theorem, for more details see \cite{Reider}.

\begin{theorem}[Reider]\label{reider}
Let $X$ be a smooth surface and let $L$ be a nef line bundle on $X$.
Then the following hold:
\begin{enumerate}

\item If $L^{2}\geq 5$ and $x\in X$ is a base point of $|K_{X}+L|$ then there exists a curve $C$ through $x$ such that either
\begin{align} \notag
& C\cdot L=0 \:\ \text{and}\:\  C^{2}=-1;\  or \\ \notag & C\cdot
L=1 \:\ \text{and}\:\ C^{2}=0.
\end{align}

\item If $L^{2}\geq 10$ and $x,y\in X$ are two (possibly infinitely near) points which $|K_{X}+L|$ does not
separate then there exists a curve $C$ through $x$ and $y$ such that
either
\begin{align} \notag
& C\cdot L=0 \:\ \text{and}\:\  C^{2}\in\left\{-1,-2\right\};\  or
\\ \notag & C\cdot L=1 \:\ \text{and}\:\
C^{2}\in\left\{0,-1\right\};\  or \\ \notag & C\cdot L=2 \:\
\text{and}\:\  C^{2}=0.
\end{align}

\end{enumerate}
\end{theorem}

The next theorem exhibits a two dimensional toroidal compactification $(X, D)$ as a smooth sub-variety of $\pp^{5}$. Moreover, it gives a bound on its degree as well as on the degree of the irreducible components of $D$.

\begin{proof}[Proof of Theorem \ref{Degree}]
Let $L:=K_{X}+D$. Because of Theorem \ref{Gap}, we know that $2K_{X}+D=2L-D$ is ample.
Theorem \ref{reider} implies us that $K_{X}+3(2L-D)=7L-4D$ is very
ample. In particular, $7L-4D$ embeds $X$ as a smooth sub-variety of
some projective space of degree $d=(7L-4D)^{2}\leq 49 L^{2}$. By
general projection, for details see page 173 in \cite{Harris}, we can always reduce to $\pp^{5}$ since $X$ is two dimensional. Next, we need to bound the degree of the components $D_{i}$ of $D$ under the embedding given by the line bundle $7L-4D$. First, recall that $D_{i}\cdot D_{j}=0$ if $i\neq j$. Then
\begin{align} \notag
d_{i}&=D_{i}\cdot (7 L-4 D) =-4 D_{i}\cdot D \\ \notag &=-4
D_{i}^{2}\leq -4 D^{2}< 6 L^{2},
\end{align}
where the last inequality is given by Corollary \ref{topsurf}.
\end{proof}

The previous theorem can be used to count the number of toroidal compactifications with fixed volume. 
We will estimate this number bounding the number of pairs in a fixed projective space of bounded degrees as in Theorem \ref{Degree}. This can be achieved studying a particular Chow variety. To ease the proof of the result, let us introduce some notation. We define $\Chow_{m}(d;d_{1},\dots, d_{q})$ to be the closed subvariety of 
\bdism
\Chow_{m}(2,d)\times \Chow_{m}(1,d_{1})\times \dots \times \Chow_{m}(1,d_{q}),
\edism 
parametrizing pairs $(X,\sum_{i=1}^{q}D_{i})$ where $X$ is a smooth surface of degree $d$ in $\pp^{m}$ and each $D_{i}$ is a smooth one-dimensional subvariety of degree $d_{i}$ in $\pp^{m}$ contained in $X$.

The following corollary is a combination of the previous results and an effective bound on the number of components of $\Chow(d:d_{1},\dots,d_{q})$. We give only a sketch of the proof since it is essentially the same as in \cite{DiCerbo}. 

\begin{corollary}\label{Count}
Fix a positive integer $V$. Then the number of toroidal
compactifications with $\dim(X)=2$ and $(K_{X}+D)^{2}\leq V$ is
bounded by 

\bdism 
\sum_{q=1}^{q_{0}}
d_{0}(6V)^{q}\binom{6d_{0}}{5}^{6(q+1)\binom{5+d_{0}}{5}}, 
\edism

where $d_{0}=49 V$ and $q_{0}=\left\lfloor 1.33 V\right\rfloor$.

\end{corollary}

\begin{proof}
We can embed $(X,D)$ in $\pp^{5}$ with bounded degrees using Theorem \ref{Degree}. Let $d$ be the degree of $X$ and let $d_{i}$ be the degree of $D_{i}$. The pair $(X,D)$ corresponds to a point of $\Chow_{5}(d;d_{1},\dots, d_{q})$. By a theorem of Fujiki in \cite{Fujiki}, $(X,D)$ is infinitesimally rigid under deformations of the pair. In particular, the number of irreducible components of $\Chow_{5}(d;d_{1},\dots,d_{q})$ provides an upper bound on the number of toroidal compactifications with fixed degrees $d$ and $d_{1},\dots,d_{q}$. It is not difficult to extend the classical bound on the number of irreducible components of Chow varieties in \cite{Kollar} to this setting, see Proposition 3.2 in \cite{Hwa2}. Combining this formula and the bounds provided by Theorem \ref{Degree} we obtain the result. 
\end{proof}

Next, we can explicitly construct the Baily-Borel compactifications of cusped complex hyperbolic surfaces as projective algebraic varieties. In other words, we give the proof of Theorem \ref{Baily1} stated in the Introduction.

\begin{proof}[Proof of Theorem \ref{Baily1}]
Let $L:=K_{X}+D$. Since
\begin{align}\notag
3\overline{c}_{2}=\overline{c}^{2}_{1}, \quad \overline{c}_{2}>0,
\end{align} 
we conclude that $L^{2}\geq 3$. Note that the null locus of $L$ is precisely $D$, i.e., the only curves that intersect trivially with the log-canonical are contained in $D$. By Reider's theorem, $K_{X}+2L$ is base point free
away from $D$. Adding $D$ we get that $3(K_{X}+D)$ is base point
free away from $D$. Next, let us consider the short exact sequence
\bdism 0\rightarrow \oo_{X}(2K_{X}+D)\stackrel{\cdot
D}{\rightarrow} \oo_{X}(2(K_{X}+D))\rightarrow \oo_{D}\rightarrow 0.
\edism 
By Kawamata-Viehweg's vanishing, see Theorem 12.2 page 181 in \cite{Van de Ven}, we have that
\begin{align}\notag
H^{1}(X,\oo_{X}(2K_{X}+D))=0.
\end{align}
Thus, taking the long exact sequence
in cohomology, there exists a surjective map \bdism
H^{0}(X,\oo_{X}(2(K_{X}+D)))\rightarrow \bigoplus_{i=1}^{q}
H^{0}(D_{i},\oo_{D_{i}}), \edism
which then implies that $2L$ is base point free along $D$.
Finally, we conclude by applying Reider's theorem to
$K_{X}+3L$.
\end{proof}

Let us conclude this section by deriving a two sided bound on the Picard number of two dimensional toroidal compactifications. More precisely, we give the proof of Theorem \ref{Picard} stated in the introduction. 

\begin{proof}[Proof of Theorem \ref{Picard}]
For any complex $4$-manifold, we have
\begin{align}\notag
p_{1}(X)=c^{2}_{1}-2c_{2}
\end{align}
where $p_{1}$ is the first Pontrjagin class. By Hodge index and
Hirzebruch's signature formula, we conclude that
\begin{align}\notag
2p_{g}(X)+2-h^{1, 1}(X)=\frac{1}{3}(c^{2}_{1}-2c_{2}),
\end{align}
where $p_{g}(X)=h^{0}(X, \mathcal{O}_{X}(K_{X}))$. As observed in the proof of Theorem \ref{Sharp}, the holomorphic Euler characteristic of $X$ is necessarily nonnegative. Noether's formula then implies that
\begin{align}\notag
h^{1, 1}(X)\leq 2p_{g}(X)+2-4\chi(\oo_{X})+c_{2}\leq
2p_{g}(X)+2+c_{2}.
\end{align}
Next, let us observe that since $D$ is
effective then
\begin{align}\notag
p_{g}(X)\leq h^{0}(X, \mathcal{O}_{X}(K_{X}+D)).
\end{align}
By a theorem of Matsusaka, see \cite{Kollar},
\begin{align}\notag
h^{0}(X, \mathcal{O}_{X}(K_{X}+D))\leq 2(K_{X}+D)^{2}+2,
\end{align}
which finally implies
\begin{align}\notag
h^{1, 1}(X)\leq 4+\frac{13}{3}(K_{X}+D)^{2}.
\end{align}
It remains to show the existence of a sharp lower bound for $\rho(X)$ in terms of the number of cusps. Say that $q$ is the number of components of $D$. Suppose by contradiction that $\rho(X)\leq q$. Let $H$ be an ample
divisor on $X$. Then there exists a relation of the form \bdism
a_{0}H+\sum_{i=1}^{q}a_{i}D_{i}\equiv 0, \edism with not all of the $\{a_{i}\}^{q}_{i=0}\in\rr$ equal to zero.
Since $D_{i}\cdot D_{j}=0$ for all $j\neq i$, and $D^{2}_{i}<0$ for all $i$, the
irreducible components of $D$ are not numerically equivalent and we can then
assume $a_{0}=1$. Intersecting the above relation with $D_{i}$, we get that $a_{i}>0$ for all $i$. On the
other hand \bdism
0=H\cdot\left(H+\sum_{i=1}^{q}a_{i}D_{i}\right)=H^{2}+\sum_{i=1}^{q}a_{i}D_{i}\cdot
H>0, \edism gives a contradiction. Finally, we need to prove that this lower bound is sharp. To this aim, let $(X, D)$ be the particular toroidal compactification used to show that Theorem \ref{Sharp} was sharp. Then $X$ is the blow up of a very particular Abelian surface $Y$ with maximal Picard number, i.e., $\rho(Y)=h^{1, 1}(Y)=4$. We then conclude that $\rho(X)=5$. Since $D$ consists of four irreducible components, the lower bound is then sharp.   
\end{proof}


\begin{thebibliography}{ELMNPM}

\bibitem[Ash et al.10]{Ash} A. Ash, D. Mumford, M. Rapoport, Y.-S. Tai, Smooth compactifications of locally symmetric varieties.  Cambridge Mathematical Library.
\textit{Cambridge University Press, Cambridge,} 2010.

\bibitem[BHPV04]{Van de Ven} W. Barth, K. Hulek, C. Peters, A. Van de Ven, Compact complex surfaces. Ergebnisse der Mathematik undihrer Grenzgebiete (3), 4. 
\textit{Springer-Verlag, Berlin,} 2004.

\bibitem[BB66]{Borel1} W. L. Baily, A. Borel, Compactifications of arithmetic quotients of bounded symmetric domains, \textit{Ann. of Math.}, \textbf{84} (1966), no.2, 442-528.

\bibitem[BJ06]{Ji} A. Borel, L. Ji, Compactifications of locally symmetric varieties.  Mathematics: Theory $\&$ Applications, Birk\"auser Boston, Inc. Boston, MA, 2006.

\bibitem[DiC12]{DiCerbo} G. Di Cerbo, L. F. Di Cerbo, Effective results for complex hyperbolic manifolds, \textit{arXiv:1212.0501} [mathDG], 2012.

\bibitem[DiC13]{Classification} L. F. Di Cerbo, Classification of toroidal compactifications with $3\overline{c}_{2}=\overline{c}^{2}_{1}$ and $\overline{c}_{2}=1$, \textit{arXiv:1309.5516} [mathAG], 2013.

\bibitem[DiC12a]{Luca1} L. F. Di Cerbo, Finite-volume complex-hyperbolic surfaces, their toroidal compactifications, and geometric applications.
\textit{Pacific J. Math.}, \textbf{255} (2012), no.2, 305-315.

\bibitem[Fuj92]{Fujiki} A. Fujiki, An $L^{2}$ Dolbeaut lemma and its applications,
\textit{Publ. Res. Inst. Math. Sci.} \textbf{28} (1992), 845-884.

\bibitem[GH78]{Harris} P. Griffiths, J. Harris,  Principles of Algebraic Geometry.
Pure and Applied Mathematics. Wiley-Interscience, New York, NY, 1978.


\bibitem[Gro82]{Gro1} M. Gromov, \textit{Volume and bounded cohomology},
Publ. Math. Inst. Hautes E\'tudes Sci. \textbf{56} (1982), 5-99.

\bibitem[Hir82]{Hir1} F. Hirzebruch, \textit{Chern numbers of algebraic surfaces: an example}, Math. Ann. \textbf{266} (1982), 351-356.

\bibitem[Hwa05]{Hwa2} J.-M. Hwang, \textit{On the number of complex hyperbolic manifolds of bounded volume},
Internat. J. Math. \textbf{8} (2005), 863-873.

\bibitem[Kol96]{Kollar} J. Koll\'ar, \textit{Rational curves on algebraic varieties}, Ergebnisse der Mathematik und ihrer Grenzgebiete. 3. Folge. A Series of Modern Surveys in Mathematics, vol. 32. Springer-Verlag, Berlin, 1996.

\bibitem[Mum77]{Mumford} D. Mumford, Hirzebruch's proportionality theorem in the non-compact case.
\textit{Invent. Math.} \textbf{42} (1977), 239-272.

\bibitem[Par98]{Parker} J. R. Parker, On the volume of cusped, complex hyperbolic manifolds and orbifolds,
\textit{Duke Math. J.} \textbf{94} (1998), 433-464.

\bibitem[Rei88]{Reider} I. Reider, Vector bundles of rank 2 and linear systems on algebraic surfaces, \textit{Ann. Math.} \textbf{127} (1988), 309-316.

\bibitem[Sto13]{Stover} M. Stover, On the number of ends of rank one locally symmetric spaces. \textit{Geom. Topol.} \textbf{17} (2013), 905-924.

\bibitem[Fri98]{Friedman} R. Friedman, \textit{Algebraic Surfaces and Holomorphic Vector Bundles},
Univesitext. Springer-Verlag, New York 1998.

\bibitem[SY82]{Siu} Y. T. Siu, S. T. Yau, Compactification of negatively curved complete K\"ahler manifolds of finite volume,
\textit{Seminars in Differential Geometry}, pp. 363-380, Ann. of
Math. Stud., Vol. 102, \textit{Princeton Univ. Press, Princeton, N.
J.}, 1982.

\bibitem[Wan72]{Wan} H. C. Wang, \textit{Topics on totally discontinuous groups}, Symmetric Spaces, 459-487,
edited by W. B. Boothby and G. L. Weiss, Pure and Appl. Math. 8,
Marcel Dekker, New York, 1972.

\end{thebibliography}
\end{document}